\documentclass[12pt]{article}

\usepackage{amssymb}
\usepackage{amsmath}
\usepackage{amsthm}
\usepackage{amsfonts}
\usepackage{mathrsfs}
\usepackage{dsfont}
\usepackage{hyperref}
\usepackage{commath}
\usepackage{graphicx}
\usepackage{text comp}
\usepackage{mathtools}

\theoremstyle{plain}
\newtheorem*{thm*}{Theorem}
\newtheorem{thm}{Theorem}[section]
\newtheorem{lemma}[thm]{Lemma}
\newtheorem*{lemma*}{Lemma}

\newtheorem*{corollary*}{Corollary}

\newtheorem*{prop*}{Proposition}
\newtheorem{defn}{Definition}

\newtheorem*{conjecture*}{Conjecture}

\newcommand{\N}{\mathbb{N}}
\newcommand{\RR}{\mathbb{R}}

\newcommand{\hhf}{{\scriptstyle{{\frac{1}{2}}}}}
\newcommand{\hof}{{\scriptstyle{{\frac{1}{4}}}}}
\newcommand{\hot}{{\scriptstyle{{\frac{1}{10}}}}}

\title{Connectivity of Random Geometric Hypergraphs}

\author{
  Henry-Louis de Kergorlay%
    \thanks{%
    School of Mathematics and Maxwell Institute for Mathematical Sciences, The University of Edinburgh, EH9 3FD
          }
 \and
        Desmond John Higham%
            \thanks{%
              School of Mathematics and Maxwell Institute for Mathematical Sciences, The University of Edinburgh, EH9 3FD
        }
      }

\begin{document}

\maketitle

\abstract{We consider a random geometric hypergraph model based on an underlying bipartite graph.
Nodes and hyperedges are sampled uniformly in a domain, and a node is assigned to those hyperedges that lie 
with a certain radius. 
From a modelling perspective, we explain how the model captures higher order connections that arise in 
real data sets.
Our main contribution is to study the connectivity properties of the model. In an asymptotic limit where the 
number of nodes and hyperedges grow in tandem we give a condition on the radius that guarantees connectivity.
}

\section{Motivation}\label{set:mot}

There is growing interest in the development of models and algorithms that capture group-level interactions
\cite{benson2016higher,bianconi2021higher,torres2020why}. For example, multiple co-authors may be involved in a collaboration, multiple workers may share an office space, and 
multiple proteins may contribute in a cellular process. In such cases representing the connectivity via a network of pairwise
interactions is an obvious, and often avoidable, simplification. Hypergraphs, where any number of nodes may be grouped together to form a hyperedge, form a natural generalization. 
Hypergraph-based techniques have been developed for 
\begin{itemize}
    \item studying the propagation of disease or information \cite{ABAMPL21,BCILLPYP21,SISonHypergraphs,higham2021epidemics,HdK22,heterogeneityHypergraph},
    \item investigating the importance or structural roles of individual components  
\cite{ER06,tudisco2022core,Tudisco2021node},
   \item discovering and quantifying clusters \cite{Chodrow21,Scholkopf2007learn,10.1214/16-AOS1453},
\item predicting future connections \cite{yoon2020much,yadati2020nhp},
\item inferring connectivity structure from time series data \cite{order23}.
\end{itemize}

Just as in the pairwise setting, it is also of interest to consider processes that create hypergraphs \cite{Bart22,Gong23}.
Comparing generative hypergraph models against real data sets may help us to understand the mechanisms 
through which interactions arise. Furthermore, realistic models can be used to produce synthetic data sets 
on which to base simulations, and also to form null models for studying features of interest.

Models that use a geometric construction, with connectivity between elements determined by distance, have proved useful in many settings.
Random geometric graphs were first introduced in \cite{Gil61} to model communication between radio
stations, although the author also mentioned their relevance to the spread of disease. 
These models have subsequently proved useful in many application areas, ranging from studies of the 
proteome 
\cite{GrindrodPeter2002Rrga,sticky2000,Fit08}
to 
academic citations \cite{cit16}.
In many settings, the notion of distance may relate to an embedding of nodes into a latent space that captures 
key features. Here, similarity is interpreted in an indirect or abstract sense.
Random geometric graphs have also been studied theoretically, with many interesting results arising from the 
perspectives of analysis, probability and 
statistical physics
\cite{OB15,Pen03,rggsurvey23,AR02,P16,Dett18}

Our aim in this work is motivate and analyse a random geometric hypergraph model.
In a similar manner to \cite{Bart22}, we make use of the connection between 
hypergraphs and bipartite graphs. The model is introduced and motivated in section~\ref{sec:model}, where we 
also show the results of illustrative computational experiments concerning connectivity. 
Our main contribution is to derive a 
condition on the thresholding radius that asymptotically guarantees connectivity of the hypergraph.
The result is stated and proved in section~\ref{sec:conn}.
Directions for future work are described in section~\ref{sec:disc}.

\section{The Random Geometric Model and Its Connectivity}\label{sec:model}

In this section we  motivate and informally describe a geometric random hypergraph model, and computationally investigate its connectivity.
We make use of a well known equivalence between hypergraphs and bipartite graphs \cite{Bart22,DGKRS21}.
Suppose we are given an undirected bipartite graph, where nodes have been separated into two groups, A and B.
By construction any edge must join one node in group A with one node in group B.
We may form a hypergraph on the nodes in group A with the following rule:
\begin{itemize}
    \item nodes in group A appear in the same hyperedge if and only if, in the underlying bipartite graph, they both have an edge to the same node in group B.
\end{itemize}
In this way the nodes in set B may be viewed as hyperedge ``centres.'' Two nodes from group A that are attracted to the same centre 
are allocated to the same hyperedge.
In many graph settings there is a natural concept of distance between nodes. For example, in social networks, geographical distance 
between places of work or residence may play a strong role in determining connectivity. More generally, there may be a 
more nuanced set of features (hobbies, tastes in music, pet ownership, \ldots) that help to explain 
whether pairwise relationships arise.
This argument extends readily to the bipartite/hypergraph scenario.
Hyperedge centres may correspond, for example, to 
shops, office buildings, gyms, train stations, restaurants, concert venues, churches,\ldots, 
with an individual joining a hyperedge if they are sufficiently close to that centre; 
for example, exercising at a local gym.
In the absence of specific information, it is natural to assume that 
the features possessed by a node arise at random, so that a node is randomly embedded in 
$\RR^d$ for some dimension $d$. 
In a similar way, we may 
simultaneously embed our hyperedge centres in 
$\RR^d$, and 
assign a node to a hyperedge if and only if it is within some threshold distance of the centre.

Figure~\ref{fig:rgh} illustrates the idea in the two dimensional case. We have a bipartite graph with two types of nodes. Groups A and B are 
represented by circles and stars, respectively.  We form a hypergraph by 
placing a circle node in a hyperedge if and only if 
it is within a certain distance of the corresponding star.
Colours in the figure distinguish between the different hyperedges. 
We emphasize that mathematically the resulting hypergraph consists only of the list of hypergraph nodes and
hyperedges. Information about the existence/number of hyperedge centres and the locations of all nodes in $\RR^2$ is lost.

\begin{figure}
\includegraphics[width=8 cm]{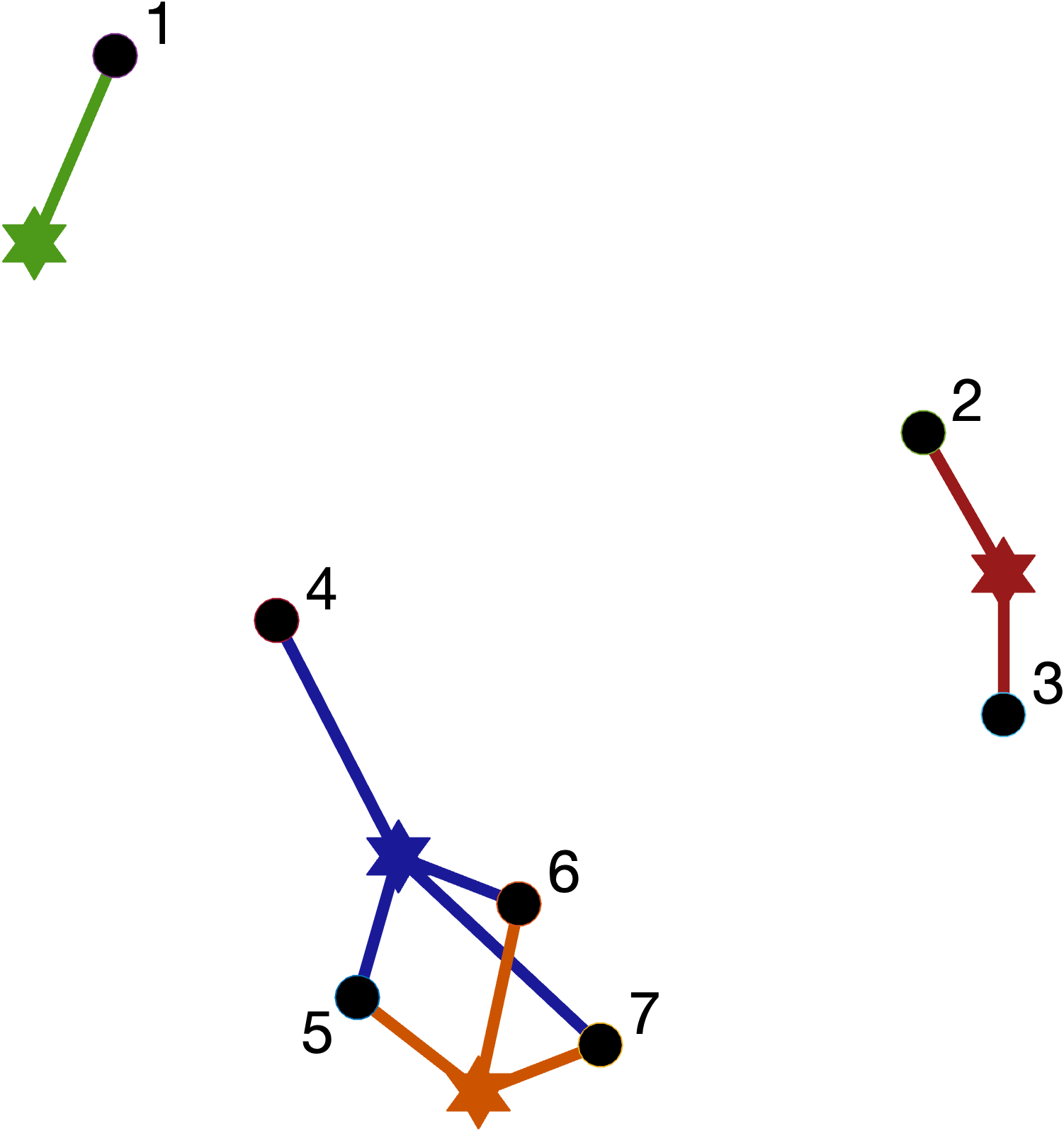}
\caption{When this construction is regarded as a bipartite graph, the solid circles and solid stars represent two types of node.
Edges are created only between nodes of a different type; this happens if any only if they are close enough in Euclidean distance.
When regarded as a hypergraph, the solid circles represent nodes and the solid stars represent ``centres'' of hyperedges.
A node is a member of a hyperedge if and only if it is sufficiently close to the corresponding centre.
Mathematically, 
the resulting hypergraph may defined by labelling the nodes $\{1,2,3,4,5,6,7\}$ and 
listing the hyperedges as 
$\{1\}$,
$\{2,3\}$,
$\{4,5,6,7\}$ and
$\{5,6,7\}$.
\label{fig:rgh}}
\end{figure}

Our aim in this work is to study connectivity: a basic property that is of practical importance in many areas, including 
disease propagation, communication and percolation.
We consider the random geometric hypergraph to be connected if the underlying random geometric bipartite graph is connected.
We focus on the smallest distance threshold that produces a connected network and study an asymptotic limit  
where the number of nodes tends to infinity.

We motivate our analytical results with computational experiments.
To produce Figure~\ref{fig:d2}, we 
formed random geometric bipartite graphs based on $n$ points embedded in $\RR^2$. 
For each graph, the points
had components chosen uniformly and independently in $(0,1)$.
We separated these points into two groups of size $n_1 = 0.8 n$ and $n_2 = 0.2 n$.
We then used a bisection algorithm to compute the smallest radius $r$ that produced
a connected bipartite graph.
In other words, we found the smallest $r$  
such that a connected graph arose when we 
created edges between pairs of 
nodes from different groups that were separated by Euclidean distance less than $r$.
(Equivalently, we
assigned $n_1 = 0.8n$ points to the role of nodes in a random geometric hypergraph 
and $n_2 = 0.2n$ points to the role of hyperedge centres, and computed the smallest 
node-hyperedge centre radius that gave connectivity.)
We ran the experiment for a range of $n$ values between
$10^3$ and $10^4$. For each choice of $n$ we repeated the computed for 
$500$ independent random node embeddings.
Figure~\ref{fig:d2} shows the mean, maximum, and minimum radius arising for each $n$.
Note that the axes are scaled logarithmically.
We have superimposed a reference line of the form $C n^{-\hhf}$, which is seen to be 
consistent with the behaviour of the radius.

\begin{figure}
\includegraphics[width=8 cm]{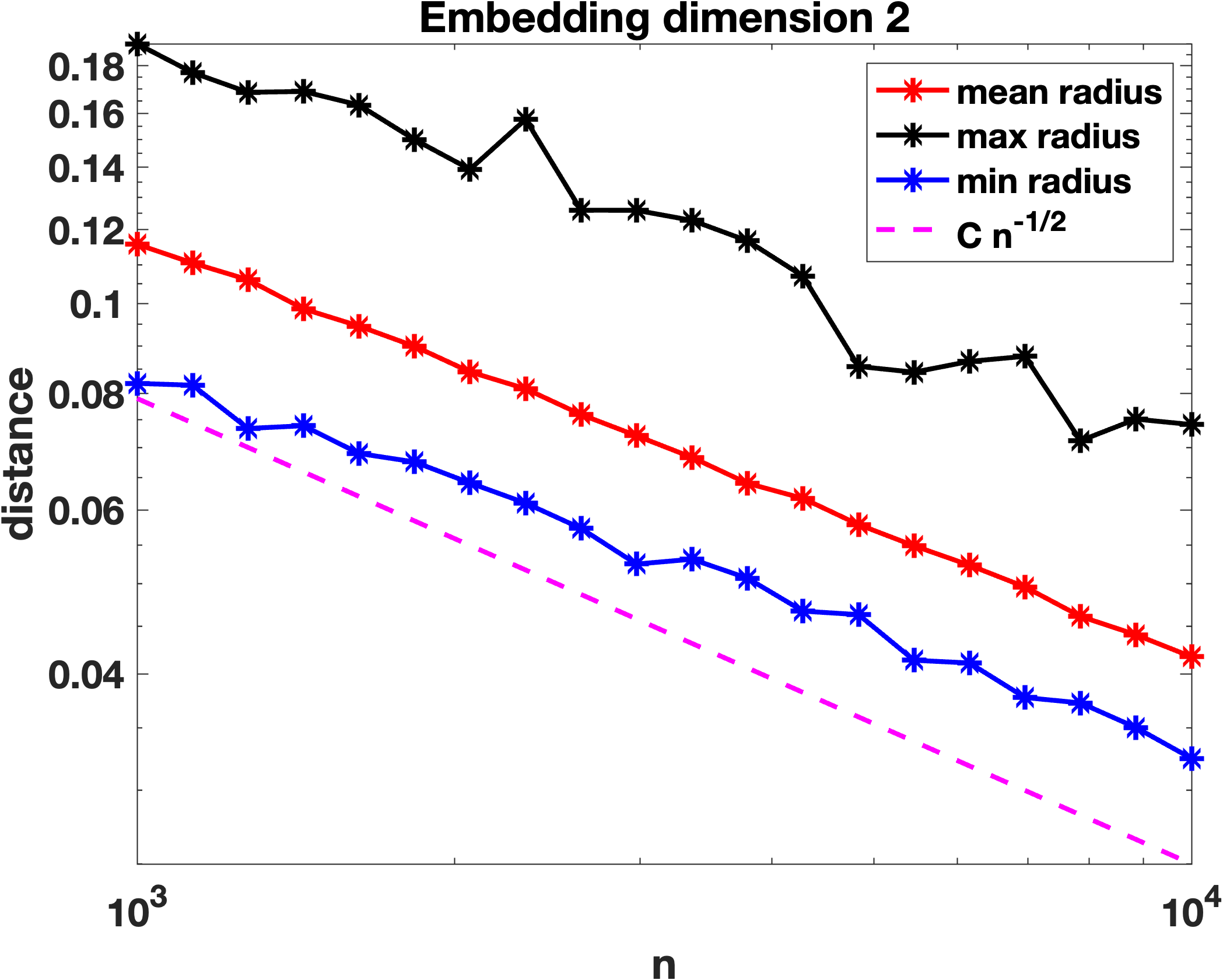}
\caption{Euclidean distance at which geometric random hypergraph becomes connected.
Here we have $0.8n$ nodes and $0.2n$ hyperedge centres in $\RR^2$, for values of $n$
between $10^3$ and $10^4$. 
 The plots show the mean, maximum and minimum value of this distance over 
$500$ independent trials.
A reference slope corresponding to $C n^{-\hhf}$ is shown. Axes are logarithmically scaled.
\label{fig:d2}}
\end{figure}  

Figures~\ref{fig:4d} and \ref{fig:10d} repeat these computations with points embedded into 
 $\RR^{4}$ and $\RR^{10}$, respectively. We see that the behaviour remains remains consistent with
 a decay roughly proportional to, and perhaps slightly slower than, $n^{-1/d}$ for dimension $d$. 

 In the next section we formalize our definition of a random geometric hypergraph and establish an 
 upper bound on the radius decay rate for connectivity that agrees with $n^{-1/d}$, up to
 $\log$-dependent factors (which of course would be extremely difficult to pin down in computational experiments).
 We also note for comparison that a threshold of the form $( \log(n)/n)^{1/d}$ has previously arisen in the study of 
 random geometric graphs,
 \cite{AR02,Pen97}.

\begin{figure}
\includegraphics[width=8 cm]{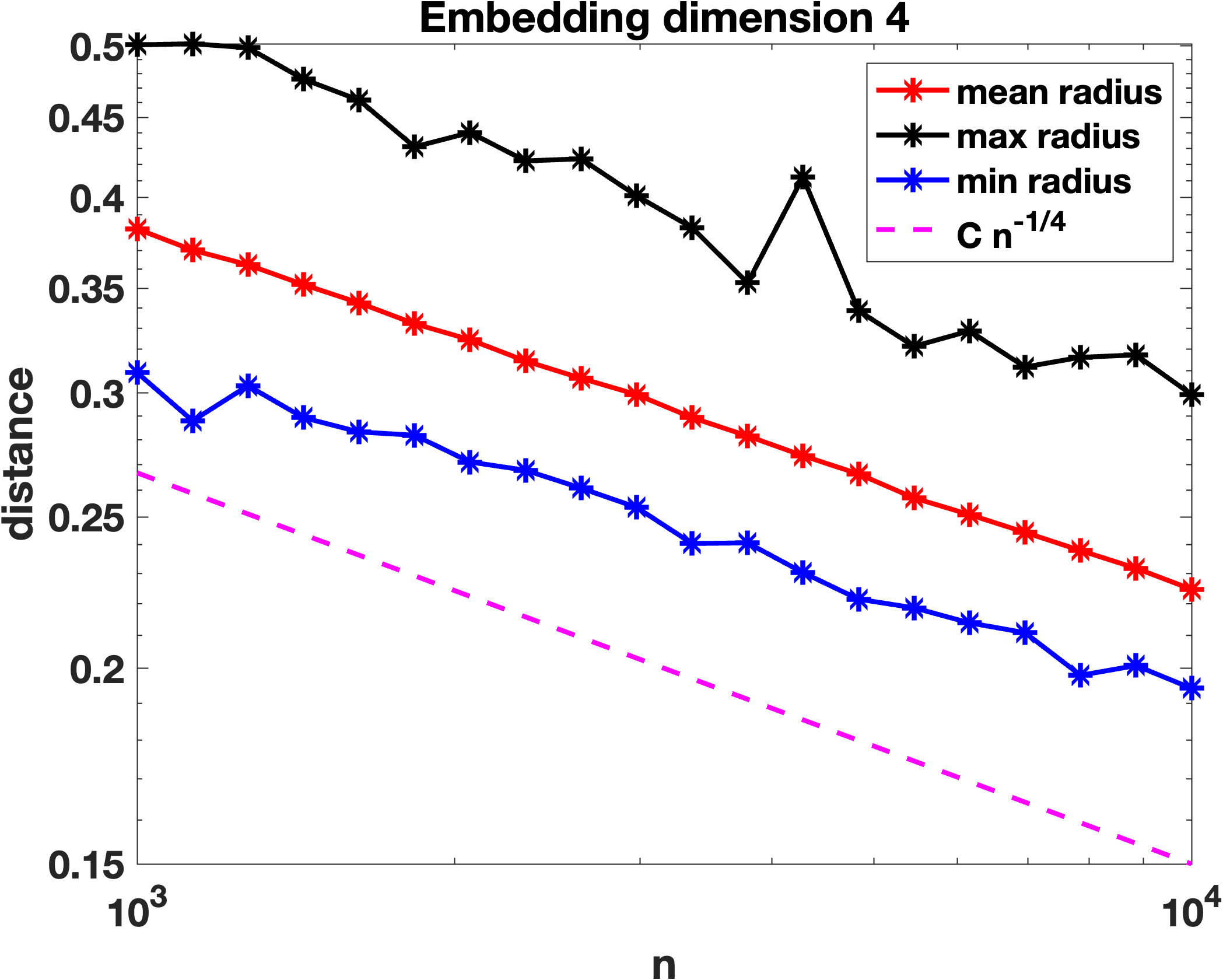}
\caption{
As for Figure~\ref{fig:d2} with nodes embedded in $\RR^4$ and a reference slope corresponding to $C n^{-\hof}$.
\label{fig:4d}}
\end{figure}

\begin{figure}
\includegraphics[width=8 cm]{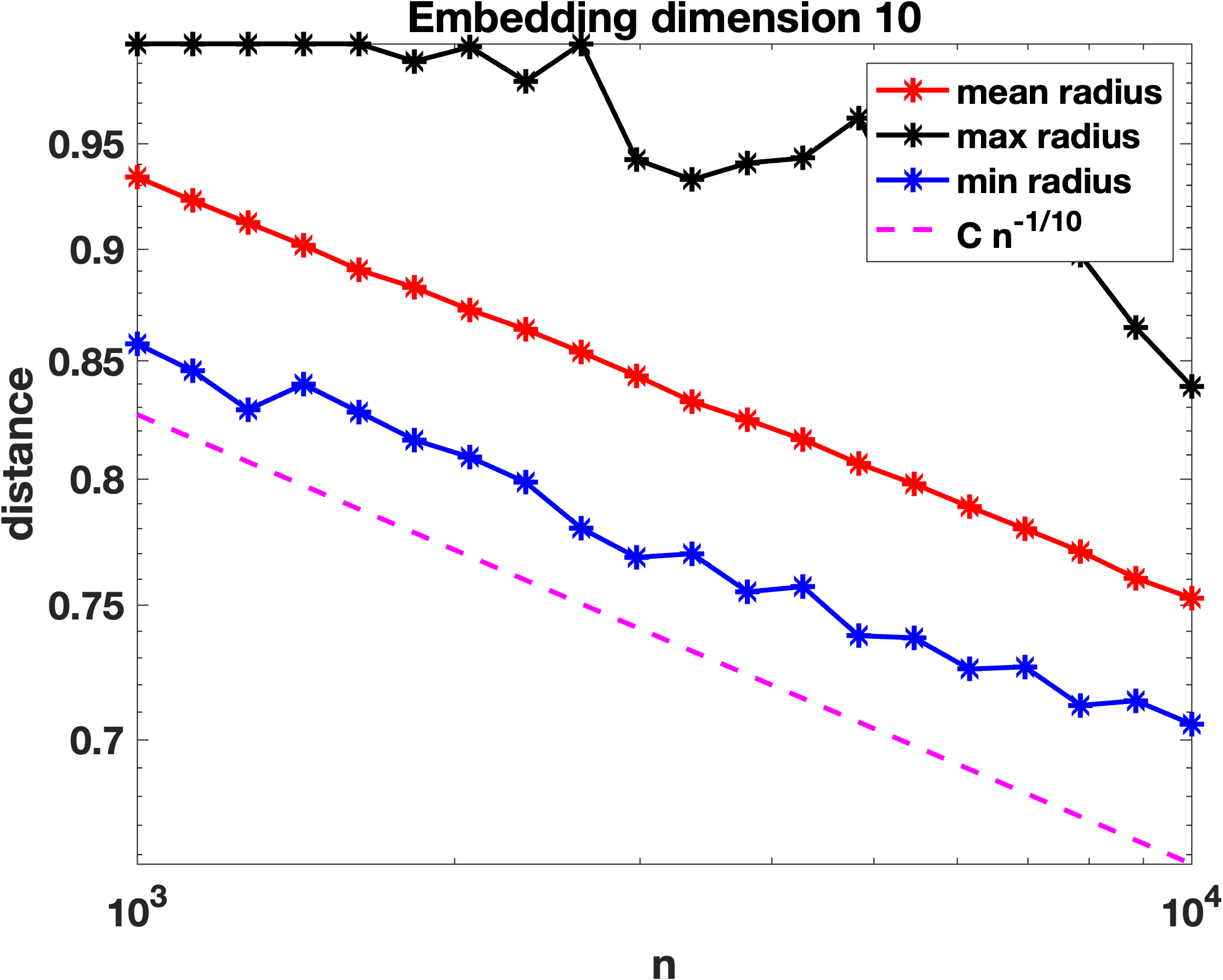}
\caption{
As for Figure~\ref{fig:d2} with nodes embedded in $\RR^{10}$ and a reference slope corresponding to $C n^{-\hot}$.
\label{fig:10d}}
\end{figure}

In related work, we note that Barthelemy \cite{Bart22} proposed and studied a wide class of random hypergraph models, including examples
where nodes are embedded in space and connections arise via a distance measure.
That approach
 to defining a random geometric hypergraph differs from ours 
 by assuming that the number of hyperedges is given and 
 by considering a process where new nodes are added to the network, with 
 new connections arising based in the current hyperedge memberships (\cite[Figure~6]{Bart22}).

\section{Connectivity Analysis}\label{sec:conn}

We now give a formal definition of a random 
geometric hypergraph and show that under reasonable conditions a thresholding radius of order 
$( \log(n)/n)^{1/d}$ ensures connectivity, asymptotically.

Let $D$ be a bounded Euclidean domain in $\RR^d$, such that $D$ has Lipschitz boundary. Given $n\in \N$, we let $\mathcal P_n$ be a Poisson point process sampled from $D$ with respect to some continuous and bounded distribution $f$, such that $f>0$ everywhere on $D$. 
We use $| \cdot|$ to denote the Euclidean norm.
 Let $n\in \N$, and let $n_1$ be the expeted number of nodes and $n_2$ be the expected number of hyperedges, chosen such that $n=n_1+n_2$. Let $r_n$ be a function of $n$, tending to $0$ as $n\to \infty$.

\begin{defn}
Let $G(n_1,n_2,r_n)$ be the probability space on the set of geometric hypergraphs, where the random nodes are chosen as a Poisson point process $\mathcal P_{n_1}$ in $D$ sampled with respect to $f$, the random hyperedges are induced by another Poisson point process $\mathcal P_{n_2}$ in $D$ sampled with respect to $f$, and where,
using bipartite graph-hypergraph equivalance,
a node $x\in \mathcal P_{n_1}$ and a hyperedge $y\in \mathcal P_{n_2}$ are connected by an edge if $\abs{x-y}<r_{n}$.
\end{defn}

Suppose that the expected number of nodes $n_1$ and of hyperedges $n_2$ satisfy 
\begin{equation*}\label{eq: choice of n_1 and n_2}
  \frac{n_1}{n_2}=\Theta(1).  
\end{equation*}
Equivalently, this means that $n_1$ and $n_2$ as functions of $n$ satisfy
$$
    n_1=\Theta(n),\quad n_2=\Theta(n).
$$
Let $K>0$ be the smallest constant such that for all $n\in \N$,
\begin{equation}\label{eq: conditions on n1 and n2}
n_1\geq \frac{1}{K}n \quad \text{ and } \quad n_2\geq \frac{1}{K}n.
\end{equation}

 Partition $\RR^d$ into a grid of cubes $\{C_{i,n}\}_i$ of width $\gamma r_n$, where $r_n=o(1)$ and $\gamma >0$ is to be determined. Let $S_n:=\{i\ |\ C_{i,n}\subset D\}$, and for each $i\in S_n$, let $\mathcal I(i,n):=\{j\not \in S_n\ |\ C_{j,n}\text{ is adjacent to } C_{i,n}\},$ and let
$$
Q_{i,n}:=\cup_{j\in \mathcal I(i,n)}(C_{j,n}\cap D).
$$

Since $D$ has a Lipschitz boundary, by compactness there exists $C>0$ depending on $D$ and $d$ (but not on $\gamma$), such that we can choose $n_0\in \N$ sufficiently large such that for all $n\geq n_0$ and all $i\in S_n$
$$
\forall\ x,y\in Q_{i,n},\ \abs{x-y}<C \gamma r_n.
$$
We then choose $\gamma:=\frac{1}{C},$ so that 
for all $i\in S_n$,
\begin{equation}\label{ineq: cube max pairwise distance}
\forall\ x,y\in Q_{i,n},\ \abs{x-y}< r_n.
\end{equation}

Note also that we have $\nu(Q_{i,n})\geq \nu(C_{i,n})\geq f_{\min}\gamma^d  r_n^d=\frac{f_{\min}}{C^d}r_n^d,$ where $f_{\min}:=\min\{f(x)\ |\ x\in \Omega\}.$ 
\begin{lemma}[Asymptotic coverage]\label{lemma: existence weak version}
Suppose that $m$ as a function of $n$ satisfies, for all $n\in\N$,
$$
m\geq \frac{1}{K}n,
$$
and suppose that $r_n$ satisfies
\begin{equation}\label{eq: rn weak version}
n\frac{f_{\min}}{KC^d} r_n^d\geq \log n-\log\log n +w(n),
\end{equation}
where $w(n)\to \infty$ arbitrarily slowly as $n\to \infty$.
With probability tending to $1$ as $n\to \infty$, for all $i\in S_n$
$$
\mathcal P_m\cap Q_{i,n}\neq \emptyset.
$$
\end{lemma}
\begin{proof}
It suffices to show that the RHS in
$$
\mathbb P(\exists\ i\in S_n, \mathcal P_m(Q_{i,n})=0)\leq \sum_{i\in S_n}\mathbb P(\mathcal P_n(Q_{i,n})=0)
$$
tends to $0$ as $n\to \infty$.

Since $\mathcal P_m$ is a homogeneous Poisson point process, we have, using (\ref{eq: rn weak version}), for all $i\in S_n$

$$
\mathbb P(\mathcal P_m(Q_{i,n})=0) = \exp(-m \nu(Q_{i,n}))\leq \exp\left(-n\frac{f_{\min}}{KC^d}  r_n^d\right) \leq n^{-1} (\log n) e^{-w(n)},
$$
and by the pigeonhole principle, $\abs{S_n}\lesssim (\gamma r_n)^{-d}\lesssim n (\log n)^{-1}$.
Hence
\begin{align*}
    \mathbb P(\exists\ i\in S_n, \mathcal P_n(Q_{i,n})=0)&\lesssim e^{-w(n)}.
\end{align*}
\end{proof}

Note that with our choice of $\gamma$, we have $$
\{\forall\ i\in S_n\ |\ \mathcal P_m\cap Q_{i,n}\neq \emptyset\}\subset \{D \subset \cup_{x\in \mathcal P_m} B(x,r_n)\}.
$$
Hence Lemma~\ref{lemma: existence weak version} gives us a lower bound estimate on the decay of $r_n$ as a function of $n$, to ensure that the balls centered at the points of $\mathcal P_m$ and of radius $r_n$, tend to form a covering of the domain $D$ as $n\to \infty$. This is an asymptotic result.

From a practical point of view, it is more useful to have a non-asymptotic version of Lemma~\ref{lemma: existence weak version}, even if we must increase slightly the constraint on the decay of $r_n$. This is the object of Lemma~\ref{lemma: existence strong version}.

\begin{lemma}[Non-asymptotic coverage]\label{lemma: existence strong version}
Suppose that $m$ as a function of $n$ satisfies, for all $n\in\N$,
$$
m\geq \frac{1}{K}n,
$$
and suppose this time that $r_n$ satisfies
\begin{equation}\label{eq: rn strong version}
n\frac{f_{\min}}{KC^d} r_n^d\geq 2\log n+\epsilon \log\log n,
\end{equation}
for some fixed $\epsilon>0$, then a.s., there exists $N\in \N$ such that for all $n\geq N$ and all $i\in S_n$
$$
\mathcal P_m\cap Q_{i,n}\neq \emptyset.
$$
\end{lemma}
\begin{proof}
A proof proceeds similarly to that of Lemma~\ref{lemma: existence weak version}, but the different constraint on $r_n$ yields instead
$$
\mathbb P(\exists\ i\in S_n, \mathcal P_m(Q_{i,n})=0)\lesssim \frac{1}{n(\log n)^{1+\epsilon}},
$$
and the required result then follows by the Borel-Cantelli lemma, since then, the series
$$
\sum_{n=0}^N\ \mathbb P(\exists\ i\in S_n, \mathcal P_m(Q_{i,n})=0)
$$
converges as $N\to \infty$.
\end{proof}
We believe that the lower bound condition on the decay of $r_n$ found in Lemma~\ref{lemma: existence weak version} is sharp, and that the lower bound condition in Lemma~\ref{lemma: existence strong version} is close to being sharp.
In Theorem~\ref{thm:conn} we apply Lemmas~\ref{lemma: existence weak version} and \ref{lemma: existence strong version} to obtain a sufficient lower bound condition on $r_n$ for the connectivity of random geometric hypergraphs, with an extra factor of $2$. We suspect that this factor could be reduced by more sophisticated analysis. 
\begin{thm}\label{thm:conn}
For every $n\in\N$, let $(n_1,n_2)\in \N^2$ satisfy (\ref{eq: conditions on n1 and n2}) and $n=n_1+n_2$. 
\begin{itemize}
\item If $r_n$ satisfies (\ref{eq: rn weak version}), then with probability tending to $1$ as $n\to \infty$, the random graph $G(n_1,n_2,2r_n)$ is connected.\\
\item If $r_n$ satisfies (\ref{eq: rn strong version}), then a.s.\ there exists $N\in \N$, such that for all $n\geq N$, the random geometric bipartite graph $G(n_1,n_2,2r_n)$ is connected.
\end{itemize} 
\end{thm}
\begin{proof}
The result is a consequence of Lemmas~\ref{lemma: existence weak version} and \ref{lemma: existence strong version} and the triangle inequality.

Suppose that $n\in \N$ is such that for all $i\in S_n$
\begin{equation}\label{eq: ideal eq}
\mathcal P_{n_1}\cap Q_{i,n}\neq \emptyset \quad \text{ and }
\quad \mathcal P_{n_2}\cap Q_{i,n}\neq \emptyset.
\end{equation}
Given two points $x,y\in \mathcal P_{n_1}$, we can find a path of adjacent cubes from $\mathcal Q_n$, such that the first cube contains $x$ and the last cube contains $y$. From (\ref{ineq: cube max pairwise distance}) and the triangle inequality, the distance between a point in one cube and another point in an adjacent cube is at most $2r_n$. Since for each cube in the path we can find a point from $\mathcal P_{n_1}$ and a point from $\mathcal P_{n_2}$, we can then form a path of edges of length at most $2r_n$ from $x$ to $y$, alternating between points in $\mathcal P_{n_1}$ and points in $\mathcal P_{n_2}$, and such a path is then a path in $\mathcal G(n_1,n_2,2r_n).$ 

This shows connectivity of the graph for all $n$ satisfying condition (\ref{eq: ideal eq}). 

This condition is true with probability tending to $1$ as $n\to \infty$, if we assume that $r_n$ satisfies (\ref{eq: rn weak version}), using Lemma~\ref{lemma: existence weak version} with $n_1$ and $n_2$ instead of $m$, giving us the first part of the theorem.

Using Lemma~\ref{lemma: existence strong version} with $n_1$ and $n_2$ instead of $m$, if $r_n$ satisfies (\ref{eq: rn strong version}), there exists $N\in \N$ such that (\ref{eq: ideal eq}) is true for all $n\geq N$, giving us the second 
part of the theorem.
\end{proof}
\section{Discussion}\label{sec:disc}

There are a number of promising avenues for further work in this area.
From a theoretical perspective, it would be of interest to derive  
useful upper bounds or indeed sharp expressions for the exact connectivity radius threshold
associated with this class of random geometric hypergraphs.
More general hypergraph models could also be developed and studied, for example using a softer version 
of the distance cut-off that has been considered in the graph setting \cite{Dett18,P16}.

From a more practical viewpoint, the related inverse problem is both challenging and potentially useful: given
a data set that corresponds to a hypergraph, 
for the model considered here 
what is the best choice of 
(a) embedding dimension, (b) node locations, and (c) hypergraph centre locations?
A similar question was addressed in \cite{Gong23} for a different generative random hypergraph model based on the assumption that  
nodes are located in a latent space and hyperedges arise preferentially between nearby nodes (without the concept of hyperedge centres).
This challenge also leads into the model selection question: 
given a data set and selection of hypergraph models, which one best describes the data, and what insights arise?

\bigskip
\noindent
\textbf{Acknowledgement}
Both authors were supported by Engineering and Physical Sciences Research Council grant EP/P020720/1. The second author was also supported by Engineering and Physical Sciences Research Council grant EP/W011093/1.

\bibliographystyle{plain}
\bibliography{references}

\end{document}